\documentclass[sn-mathphys-num]{sn-jnl}

\usepackage{graphicx}%
\usepackage{multirow}%
\usepackage{amsmath,amssymb,amsfonts}%
\usepackage{amsthm}%
\usepackage{mathrsfs}%
\usepackage[title]{appendix}%
\usepackage{xcolor}%
\usepackage{textcomp}%
\usepackage{manyfoot}%
\usepackage{booktabs}%
\usepackage{algorithm}%
\usepackage{algorithmicx}%
\usepackage{algpseudocode}%
\usepackage{listings}%
\newcommand{\mylabel}[2]{#2\def\@currentlabel{#2}\label{#1}}

\newcommand{\C}{\mathcal{C}}

\DeclareMathOperator*{\n}{n}

\DeclareMathOperator{\R}{\mathbb{R}}

\DeclareMathOperator*{\rd}{rd}

\makeatletter
\newcommand{\pushright}[1]{\ifmeasuring@#1\else\omit\hfill$\displaystyle#1$\fi\ignorespaces}
\newcommand{\pushleft}[1]{\ifmeasuring@#1\else\omit$\displaystyle#1$\hfill\fi\ignorespaces}
\makeatother
\allowdisplaybreaks

\makeatletter
\def\Ddots{\mathinner{\mkern1mu\raise\p@
\vbox{\kern7\p@\hbox{.}}\mkern2mu
\raise4\p@\hbox{.}\mkern2mu\raise7\p@\hbox{.}\mkern1mu}}
\makeatother

\theoremstyle{thmstyleone}%
\theoremstyle{thmstyletwo}%

\theoremstyle{thmstylethree}%

\newtheorem{thm}{Theorem}[section]

\newtheorem{lem}[thm]{Lemma}
\newtheorem{cor}[thm]{Corollary}
\theoremstyle{remark}
\newtheorem{rem}[thm]{Remark}
\newtheorem{defn}[thm]{Definition}

\newtheorem*{thm*}{Theorem}

\theoremstyle{remark}

\theoremstyle{definition}

\usepackage{graphicx}

\begin{document}

\title[Caterpillars with given degree sequence, small Energy and small Hosoya index]{Caterpillars with given degree sequence, small Energy and small Hosoya index}


\author*[1]{\fnm{Eric O. D.} \sur{Andriantiana}}\email{e.andriantiana@ru.ac.za}
\equalcont{These authors contributed equally to this work.}

\author[2]{\fnm{Xhanti} \sur{Sinoxolo}}\email{xhantma@gmail.com }
\equalcont{These authors contributed equally to this work.}

\affil*[1,2]{\orgdiv{Mathematics (Pure \& Applied)}, \orgname{Rhodes University}, \orgaddress{\street{Somerset Street}, \city{Grahamstown}, \postcode{6139}, \state{Eastern Cape}, \country{South Africa}}}


\abstract{The energy $En(G)$ of a graph $G$ is defined as the sum of the absolute values of its eigenvalues. The Hosoya index $Z(G)$ of a graph $G$ is the number of independent edge subsets of $G$, including the empty set.
For any given degree sequence $D$, we characterize the caterpillar $\mathcal{S}(D)$ that has the minimum $Z$ and $En$. 
We also compare $\mathcal{S}(D)$ with $\mathcal{S}(Y)$, for a degree sequence $Y$ majorized by a degree sequence $D$. Suppose $Y=(y_1,\dots ,y_n)$ and $D=(d_1,\dots ,d_n)$ are degree sequences such that $Y$ is majorized by $D$ and$$\sum_{i=1}^{n}y_i=\sum_{i=1}^{n}d_i,$$then $Z(\mathcal{S}(D))<Z(\mathcal{S}(Y))$ and $En(\mathcal{S}(D))<En(\mathcal{S}(Y))$. 
}




\maketitle

\section{Introduction}

For a given graph $G=(V(G),E(G))$, with vertex set $V(G)$ and edge set $E(G)$, a subset $M$ of $E(G)$ is called a matching of $G$ if no two edges in $M$ are adjacent in $G$. The edges in $M$ are said to be independent. The two ends of an edge in a matching $M$ of a graph $G$ are said to be matched under $M$. 
Note that $H$ can be empty. The number of matchings of order $k$ of $G$ is denoted by $m(G,k)$. 

Let $G$ be a graph of order $n$, $A(G)$ the adjacency matrix of $G$, the sum of the absolute values of the eigenvalues of $A(G)$ is called the energy of $G$, and is denoted by $En(G)$. The following well-known Coulson integral formula for the energy of forests allows one to study $En (G)$ using $m(G,k)$. 

\begin{thm}\cite{[IGUTMANOEPOLANSKYBOOKMATHEMATICALCONCEPTS]}
If $G$ is a forest of order $n$, then 
\begin{equation}\label{energyoftrees1}
En(G)
=\frac{2}{\pi}\int_{0}^{+\infty}\frac{1}{x^2}\ln\left[\sum_{k\geq0}m(G,k)x^{2k}\right]dx,
\end{equation}
where $m(G,k)$ is the matching number of order $k$.
\end{thm}

With $$M(G,x)=\sum_{k\geq0}m(G,k)x^{k},$$
we have 
\begin{equation}\label{energyoftrees2}
En(G)=\frac{2}{\pi}\int_{0}^{+\infty}\frac{1}{x^2}\ln M(G,x^{2})dx.
\end{equation}
It is clear from Equation \eqref{energyoftrees2}, that if $T$ and $T^{'}$ are trees and $M(T,x)\leq M(T^{'},x)$, for all positive $x\in\R$, then $En(T)\leq En(T^{'})$. If furthermore, there exists a positive $x\in\R$, such that $M(T,x)<M(T^{'},x)$, then $En(T)<En(T^{'})$.

The Hosoya index $Z(G)$ of a graph $G$ is defined \cite{[9HHOSOYA]} as the total number of matchings in $G$, that is,
\begin{equation}\label{HosoyaindexOriginalEquation}
Z(G)=\sum_{k\geq0}m(G,k).
\end{equation}
Hence, $Z(G)=M(G,1)$. If $G$ and $G^{'}$ are two graphs, such that for all positive $x\in\R$, $M(G,x)\leq M(G^{'},x)$, then $Z(G)\leq Z(G^{'})$. 
 
In addition to purely mathematical interest, the study of the energy and the Hosoya index is motivated by its applications in chemistry. See the book \cite{[BOOKGRAPHENERGY]} for more details.

The non-increasing sequence $D=(d_1,d_2,\dots ,d_n)$ of vertex degrees of a graph $G$ is called the degree sequence of $G$. 
The reduced degree sequence of a graph $G$ is a non-increasing sequence obtained by removing all the $1$'s in the degree sequence of $G$. The set of all trees with degree sequence $D$ will be denoted by $\mathbb{T}(D)$.
$\mathbb{T}(D)$ is a well studied class of graphs relative to various graph invariants. See for example \cite{[1],[2],[385],[69W],[316],[344],Andriantiana20211}. 

For $n\geq 1$, let $a_1,a_2,\dots ,a_n$ be non-negative integers, such that $a_1$ and $a_n$ are positive, the tree obtained from the path graph, with vertices $v_1,v_2,\dots ,v_n$ in this order from one end to the other, by attaching $a_i$ new leaves to $v_i$, for $1\leq i\leq n$, is called a $(a_1,a_2,\dots ,a_n)$-caterpillar, and is denoted by $C(a_1+1,a_2+2,\dots ,a_{n-1}+2,a_n+1)$. The set of all caterpillars with degree or reduced degree sequence $D$ is denoted by $\mathbb{C}_{D}$. 

In \citep{Andriantiana20211}, a characterisation of the tree with given degree sequence, minimum energy and minimum Hosoya index is provided. Theorems in \citep{Andriantiana20211} can be used to various classes of trees such as $n$-vertex trees with maximum degree or given number of vertices of degree $1$. Unfortunately, they are not enough to provide the structure of caterpillars with degree sequence $D$ and minimum energy or minimum Hosoya index. In this paper, we modify techniques used in \citep{Andriantiana20211,Heuberger2009214} to be suitable for caterpillars, and use it to obtain characterisation of the caterpillar $\mathcal{S}(D)$ with degree sequence $D$ which has the minimum Hosoya index and minimum energy. A preliminary section for technical definitions and lemmas is provided before the main section, where the structure of $\mathcal{S}(D)$ is described. In  the last section, we compare $\mathcal{S}(D)$ and $\mathcal{S}(D')$ for different $D$ and $D'$ and obtained corollaries on caterpillars with various degree restrictions.
 

\section{Preliminaries}
Technical lemmas and terminology in this section will play crucial roles in the rest of the paper.

Let $G$ be a 
graph. The neighbourhood of $v$ in $G$ is denoted by $N_{G}(v)$, and $N{(v)}$ when there is no necessity to specify the graph $G$. The set of vertices $v$ and its neighbours is denoted by $N_{G}[v]$ or $N[v]$ (i.e $N_{G}[v]=\{v\}\cup N{(v)}$). The degree of $v$ is defined as $|N(v)|$ and denoted by $\deg(v)$.

We say $G$ is minimal with regard to an invariant $f(G)$, if and only if $f(G)=\min\{f(C):C\in\mathbb{C}_{D}\}$. 


The subtree $B$ of a tree $T$ is called a complete branch of $T$ if and only if $T-V(B)$ is connected.
When a complete branch $B$ of $T$ is considered as a rooted tree, we choose its vertex that has a neighbour in $T-B$ to be the root, and we denote it by $r(B)\in V(B)$. The degree of $r(B)$ is denoted by $\rd(B)$. Note that the degree of the vertex $r(B)$ in $T$ is $rd(B)+1$. If $B_1,B_2,\dots ,B_{\rd(B)}$ are the connected components of $B-r(B)$, we write $B=[B_1,B_2,\dots ,B_{\rd(B)}]$. For any two complete branches $B_1$ and $B_2$ we write $B_1\approx _rB_2$ if and only if there exists a graph isomorphism $f:V(B_1)\rightarrow V(B_2)$ which preserves the roots, that is $f(r(B_1))=r(B_2)$. 

\begin{defn}
A non-leaf vertex in a tree $T$, that has at most one neighbour of degree greater than $1$ is called a pseudo-leaf. A complete branch $B$ is called a pseudo-leaf branch if its root is a pseudo-leaf. A pseudo-leaf branch with $d$ vertices is denoted by $[d]$.
\end{defn}


Recall the following well-known properties of $M(G,x)$. See \cite{[EODANDRIANTIANA]} for a proof. 
\begin{lem}
\label{lemma1p}
Let $G$ and $G^{'}$ be two disjoint graphs and let $x>0$ be a real number. Then, we have \begin{equation} M(G\cup G^{'},x)=M(G,x)M(G^{'},x).\label{3.17}\end{equation}If $v\in V(G)$, then we have \begin{equation}
M(G,x)=M(G-v,x)+x\sum_{w\in N_G(v)}M(G-\{v,w\},x)\label{3.18}.
\end{equation}For any $uv\in E(G)$ we have \begin{equation}
M(G,x)=M(G-uv,x)+xM(G-\{u,v\},x).\label{3.19}
\end{equation}
\end{lem}

For every complete branch $B$ of a tree, we define $m_0(B,k)$ to be the number of matchings of order $k$ in $B$, not covering the root $r(B)$, $M_0(B,x)=\sum_{k\geq0}m_0(B,k)x^k$, and $\tau(B,x)=\frac{M_0(B,x)}{M(B,x)}$. Repeated use of Lemma \ref{lemma1p} leads to the following lemma.

\begin{lem}[\cite{[1]}]\label{lemma2p}
Let $B=[B_1,\dots ,B_{\rd(B)}]$ be a complete branch of a tree. Then, for all positive $x$ we have \begin{equation}\displaystyle{\tau (B,x)=\frac{1}{1+x\sum_{i=1}^{\rd(B)}\tau (B_i,x)}}.\label{6}\end{equation}It is convenient to set $\tau (\emptyset ,x)=0$ for all $x>0$, so that the recurrence formula \eqref{6} still holds if some of the $B_i$'s are empty.
\end{lem}

\begin{lem}[\cite{[1]}]\label{lemma3p}
Let $B$ be a complete branch of a tree and $x>0$. Then, $$\frac{1}{1+x\rd(B)}\leq\tau (B,x)\leq 1.$$
\end{lem}
\begin{rem}\label{remark4p}
Note that the upper bound $1$ is reached only if $B$ is a leaf. It follows that the lower bound is  obtained only if $B$ is a pseudo-leaf branch.
\end{rem}
\section{Caterpillars with given degree sequence and small $M(.,x)$, energy and Hosoya index}
In this section, 
we characterize caterpillars with given degree sequence and smallest $M(.,x)$, and hence smallest Energy and Hosoya index.

Let $G$ be a caterpillar and be decomposed as in Figure \ref{Figure-1}. Then, using Equations \eqref{3.17} and \eqref{3.19} we have:
\begin{figure}[htbp]
\centering
\includegraphics[scale=.5]{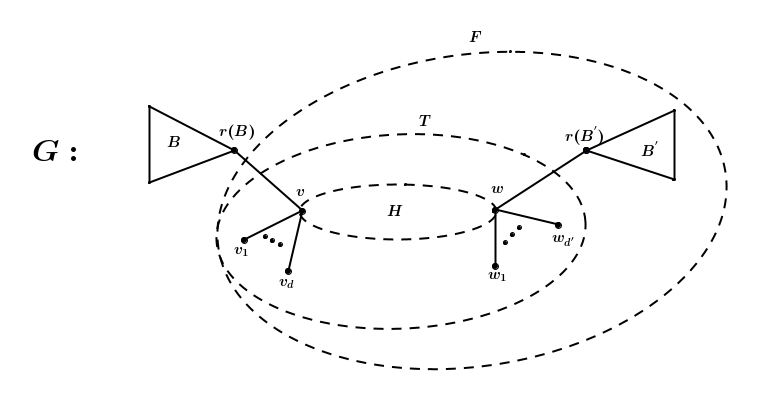}
\caption{Decomposition of a caterpillar for sections 3 and 4.}
\label{Figure-1}
\end{figure}

\begin{align*}
&M(G,x)\\
&=M(G-vr(B),x)+xM(G-\{r(B),v\},x)\\
&=M(B,x)M(F,x)+xM(G-\{r(B),v\},x)\\
&=M(B,x)\left[M(F-wr(B^{'}),x)+xM(F-\{w,r(B^{'})\},x)\right]
+xM(G-\{r(B),v\},x)\\
&=M(B,x)\left[M(T,x)M(B^{'},x)+xM(F-\{w,r(B^{'})\},x)\right]
+xM(G-\{r(B),v\},x)\\
&=M(B,x)\left[M(T,x)M(B^{'},x)+xM(T-w,x)M(B^{'}-r(B^{'}),x)\right]\\
&\pushright{+xM(B-r(B),x)M(F-v,x)}\\
&=M(B,x)\left[M(T,x)M(B^{'},x)+xM(T-w,x)M(B^{'}-r(B^{'}),x)\right]\\
&\pushright{+xM(B-r(B),x)\left[M(F-v-wr(B^{'}),x)+xM(F-v-\{w,r(B^{'})\},x)\right]}\\
&=M(B,x)\left[M(T,x)M(B^{'},x)+xM(T-w,x)M(B^{'}-r(B^{'}),x)\right]\\
&+xM(B-r(B),x)\left[M(T-v,x)M(B^{'},x)+xM(T-\{w,v\},x)M(B^{'}-r(B^{'}),x)\right]\\
&=M(B,x)M(B^{'},x)M(T,x)+xM(B,x)M(B^{'}-r(B^{'}),x)M(T-w,x)\\
&+xM(B^{'},x)M(B-r(B),x)M(T-v,x)\\
&\pushright{+x^2M(B-r(B),x)M(B^{'}-r(B^{'}),x)M(T-\{v,w\},x).}
\end{align*}
Further iterative use of Equations \eqref{3.17} and \eqref{3.19} gives

\begin{align*}
M(T,x)
&=M(T-\{v_1,\dots ,v_d\}-\{w_1,\dots ,w_{d'}\},x)+d'xM(T-\{v_1,\dots ,v_d\}-w,x)\\
&\pushright{+dxM(T-v,x)}\\
&=M(H,x)+d'xM(H-w,x)+dx\left[M(T-v-\{w_1,\dots ,w_{d'}\},x)\right.\\
&\hspace*{9.5cm}\left.+d'xM(T-v-w,x)\right]\\
&=M(H,x)+d'xM(H-w,x)+dx\left[M(H-v,x)+d'xM(H-\{v,w\},x)\right],
\end{align*}

\begin{align*}
M(T-w,x)
&=M(T-w-\{v_1,\dots ,v_d\},x)+dxM(T-w-v-\{v_1,\dots ,v_d\},x)\\
&=M(H-w,x)+dxM(H-w-v,x)\\
\end{align*}
and 
$$M(T-v,x)=M(H-v,x)+d^{'}xM(H-w-v,x).$$ Hence,
\begin{align}
&M(G,x)\nonumber\\
&=M(B,x)M(B^{'},x)\left[M(H,x)+d^{'}xM(H-w,x)+dx\left[M(H-v,x)+d^{'}xM(H-\{v,w\},x)\right]\right]\nonumber\\
&\pushright{+xM(B,x)M(B^{'}-r(B^{'}),x)\left[M(H-w,x)+dxM(H-\{v,w\},x)\right]\nonumber}\\
&+xM(B^{'},x)M(B-r(B),x)\left[M(H-v,x)+d^{'}xM(H-\{v,v\},x)\right]\nonumber\\
&\pushright{+x^2M(B-r(B),x)M(B^{'}-r(B^{'}),x)M(H-\{v,w\},x)\nonumber}\\
&=M(B,x)M(B^{'},x)\left[M(H,x)+d^{'}xM(H-w,x)\right.\nonumber\\
&\hspace*{5.7cm}\left.+dx\left[M(H-v,x)+d^{'}xM(H-\{v,w\},x)\right]\right]\nonumber\\
&+xM(B,x)M(B^{'},x)\tau(B^{'},x)\left[M(H-w,x)+dxM(H-\{v,w\},x)\right]\nonumber\\
&\hspace*{1.6cm}+xM(B,x)M(B^{'},x)\tau(B,x)\left[M(H-v,x)+d^{'}xM(H-\{v,w\},x)\right]\nonumber\\
&+x^2M(B,x)M(B^{'},x)\tau(B,x)\tau(B^{'},x)M(H-\{v,w\},x)\nonumber\\
&=M(B,x)M(B^{'},x)\left[M(H,x)+d^{'}xM(H-w,x)\right.\nonumber\\
&\hspace*{5.8cm}\left.+dx\left[M(H-v,x)+d^{'}xM(H-\{v,w\},x)\right]\right.\nonumber\\
&+x\tau(B^{'},x)\left[M(H-w,x)+dxM(H-\{v,w\},x)\right]\nonumber\\
&\hspace*{4.8cm}+x\tau(B,x)\left[M(H-v,x)+d'xM(H-\{v,w\},x)\right]\nonumber\\
&+x^2\left.\tau(B,x)\tau(B^{'},x)M(H-\{v,w\},x)\right]\nonumber\\
&=M(B,x)M(B^{'},x)\left[M(H,x)+dd^{'}x^2M(H-\{v,w\},x)\right.\nonumber\\
&+x^2\tau(B,x)\tau(B^{'},x)M(H-\{v,w\},x)
+x\left[d^{'}M(H-w,x)+dM(H-v,x)\right]\nonumber\\
&+x\left[\tau(B^{'},x)M(H-w,x)+\tau(B,x)M(H-v,x)\right]\nonumber\\
&\hspace*{5.1cm}+x^{2}M(H-\{v,w\},x)\left.\left[d\tau(B^{'},x)+d^{'}\tau(B,x)\right]\right]\nonumber\\
&=M(B,x)M(B^{'},x)\left[M(H,x)+dd^{'}x^2M(H-\{v,w\},x)\right.\nonumber\\
&\pushright{\left.+x^2\tau(B,x)\tau(B^{'},x)M(H-\{v,w\},x)\right.\nonumber}\\
&+x\left[\left(d^{'}+\tau(B^{'},x)\right)M(H-w,x)+\left(d+\tau(B,x)\right)M(H-v,x)\right]\nonumber\\
&\pushright{+x^2M(H-\{v,w\},x)\left.\left[d\tau(B^{'},x)+d^{'}\tau(B,x)\right]\right].}\label{EquationMax}
\end{align}
We define
\begin{align*}
M_{v}^{w}(G,x)
&=x\left[\left(d^{'}+\tau(B^{'},x)\right)M(H-w,x)+\left(d+\tau(B,x)\right)M(H-v,x)\right]\\
&+x^2M(H-\{v,w\},x)\left[d\tau(B^{'},x)+d^{'}\tau(B,x)\right],
\end{align*}
so that
\begin{align}
M(G,x)
&=M(B,x)M(B^{'},x)\left[M(H,x)+dd^{'}x^2M(H-\{v,w\},x)\right.\nonumber\\
&\hspace*{1.4cm}+\left.x^2\tau(B,x)\tau(B^{'},x)M(H-\{v,w\},x)+M_{v}^{w}(G,x)\right].\label{M(G,x)}
\end{align}

\subsection{Caterpillars with given degree sequence and minimum $M(.,x)$}

In this section we characterize the caterpillar with given degree sequence $D$, minimum $M(.,x)$ and hence minimum Hosoya index and minimum energy. 


The following simple technical lemma will play central role as we try to find out what exchange of branches reduces $M(.,x)$.

\begin{lem}\label{minimal0}
Let $a,b,c$ and $d$ be nonnegative real numbers. If $a\leq b$ and $c\leq d$, then $ac+bd\geq ad+bc$. If $a\leq b\leq c\leq d$, then $ab+cd\geq ac+bd\geq ad+bc$.
\end{lem}
\begin{proof}
The lemma follows from the fact that
\begin{align*}
ac+bd-(ad+bc)&=(b-a)(d-c)\quad \text{and} \quad
ab+cd-(ac+bd)=(d-a)(c-b).
\end{align*}
\end{proof}
\begin{lem}\label{lemma4m}
Let $B$ and $B^{'}$ be complete branches of a caterpillar $G$, $h(B)$ and $h(B^{'})$ their respective heights. If $h(B)=h(B^{'})$, then $\tau (B,x)=\tau (B^{'},x)$ only if $B\approx_rB^{'}$.
\end{lem}

\begin{proof}

Suppose $h(B)=h(B^{'})=h$. If $h=1$ and $\tau (B,x)=\tau (B^{'},x)$, then $$\tau (B,x)=\frac{1}{1+x\rd(B)}=\frac{1}{1+x\rd(B^{'})}=\tau (B^{'},x).$$ Then, $1+x\rd(B)=1+x\rd(B^{'})$. Since $x>0$, $\rd(B)=\rd(B^{'})$ and $B\approx_rB^{'}$. 

Suppose that for $h=k\geq 1$, $\tau (B,x)=\tau (B^{'},x)$ implies $B\approx_rB^{'}$. Now, consider the case of $h=k+1\geq 2$.
$B$ and $B^{'}$ can be decomposed as in Figure \ref{Figure-2}.
\begin{figure}[htbp]
\centering
\includegraphics[scale=.5]{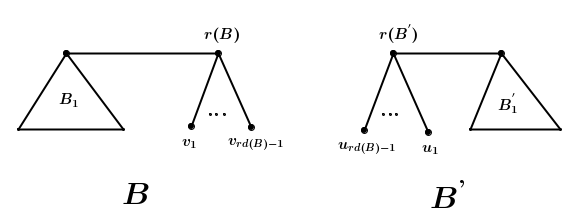}
\caption{Decomposition of complete branches in Lemma \ref{lemma4m}.}
\label{Figure-2}
\end{figure}
Then we have
 $$\displaystyle{\tau (B,x)=\frac{1}{1+x\sum_{i=1}^{\rd(B)}\tau (B_i,x)}=\frac{1}{1+x\left(\rd(B)-1+\tau (B_1,x)\right)}}$$ and $$\tau (B^{'},x)=\frac{1}{1+x\left(\rd(B^{'})-1+\tau (B_1^{'},x)\right)}.$$ From $\tau (B,x)=\tau (B^{'},x)$, we have $1+x\left(\rd(B)-1+\tau (B_1,x)\right)=1+x\left(\rd(B^{'})-1+\tau (B_1^{'},x)\right)$ and thus \begin{equation}\label{1}\rd(B)+\tau (B_1,x)=\rd(B^{'})+\tau (B^{'}_1,x).\end{equation}Suppose that $\rd(B)\neq \rd(B^{'})$. Without loss of generality, assume $\rd(B)$$>\rd(B^{'})$. Then, $\rd(B)\geq \rd(B^{'})+1$. Since  $h\geq 2$, the branches $B_1$ and $B_{1}^{'}$ are non-empty and $0<\tau (B_1,x),\tau (B^{'}_1,x)\leq 1$. Hence,
\begin{equation*}
\rd(B)+\tau (B_1,x)>\rd(B^{'})+1\geq \rd(B^{'})+\tau (B^{'}_1,x).
\end{equation*}
But this would contradict \eqref{1}. So, we must have $\rd(B)=\rd(B^{'})$. Combined with \eqref{1}, this gives  $\tau (B^{'}_1,x)=\tau (B_1,x)$ and $\tau (B^{'}_1,x)=\tau (B_1,x)$, which implies $B_1\approx_rB^{'}_1$, and thus $B\approx_rB^{'}$.
\end{proof}

\begin{lem}\label{minimal1}
Let $B$ and $B'$ be non-empty complete branches of a caterpillar $G$ with degree sequence $D$, attached at $v$ and $w$ respectively.  Suppose that $M(G,x)\leq M(T,x)$ for all $T\in\mathbb{C}_D$ for some $x>0$. Let $d$ and $d'$ be the number of leaves attached to $v$ and $w$, respectively.
\begin{itemize}
\item[i)] If $\tau(B^{'},x)>\tau(B,x)$, then $d^{'}\geq d$ and $M(H-w,x)\leq M(H-v,x)$.
\item[ii)] If $d^{'}>d$ and $\tau(B^{'},x)\neq\tau(B,x)$, then $\tau(B^{'},x)>\tau(B,x)$ and $M(H-w,x)\leq M(H-v,x)$.
\item[iii)] If $d^{'}>d$ and $\tau(B^{'},x)=\tau(B,x)$, then $M(H-w,x)\leq M(H-v,x)$.
\end{itemize}
\end{lem}
\begin{proof}
Recall that 
\begin{align*}
&M_{v}^{w}(G,x)=x\left[d^{'}M(H-w,x)+dM(H-v,x)\right]\\
&\hspace{3cm}+x\left[\tau(B^{'},x)M(H-w,x)+\tau(B,x)M(H-v,x)\right]\\
&\hspace{5cm}+x^2M(H-\{v,w\},x)\left[d\tau(B^{'},x)+d^{'}\tau(B,x)\right].
\end{align*}
Given that $\tau(B^{'},x)>\tau(B,x)$, by Lemma \ref{minimal0},
 $d\tau(B^{'},x)+d^{'}\tau(B,x)$ is smallest only if $d\leq d^{'}$; $\tau(B^{'},x)M(H-w,x)+\tau(B,x)M(H-v,x)$ is smallest only if $M(H-w,x)\leq M(H-v,x)$. By Lemma \ref{minimal0} again, the minimality of $d^{'}M(H-w,x)+dM(H-v,x)$ requires  $d^{'}\geq d$ and $M(H-w,x)\leq M(H-v,x)$. This proves i).
 
To prove ii), suppose $d^{'}>d$, $\tau(B^{'},x)\neq\tau(B,x)$. 
Suppose $\tau(B^{'},x)<\tau(B,x)$. Then, from i) we must have $d^{'}\leq d$. But $d^{'}>d$. This is a contradiction. Hence $\tau(B^{'},x)\geq\tau(B,x)$. Since $\tau(B^{'},x)\neq\tau(B,x)$, we must have $\tau(B^{'},x)>\tau(B,x)$. From i), this requires $M(H-w,x)\leq M(H-v,x)$. 

To prove iii), we proceed as in i). Since $\tau(B^{'},x)=\tau(B,x)$, then $M_{v}^{w}(G,x)$ is minimal only if the expression $d^{'}M(H-w,x)+dM(H-v,x)$ is minimal. Given that $d^{'}>d$, by Lemma \ref{minimal0}, we must have $M(H-w,x)\leq M(H-v,x)$. 
\end{proof}
\begin{lem}\label{minimal2}
Let $G$ be a caterpillar and $x$ a positive real number. Label all the non-leaf vertices in $G$ from left to right as $u_1,u_2,\dots ,u_n$. If $G$ is  minimal with respect to $M(.,x)$, then $u_1$ and $u_n$ have the largest degrees in $G$.
\end{lem}
\begin{proof}
If $n=1$, then $G$ is a star and has only one non-leaf vertex, which is of largest degree. If $n=2$, then $G$ has two non-leaf vertices and they have the largest degrees in $G$. If $n\geq 3$, decompose $G$ as in Figure \ref{Figure-1} with $B$ a leaf adjacent to $u_1$ and $B^{'}$ a complete branch of $G$ such that $B^{'}$ does not contain $u_1$ and $u_2$ and the root of $B^{'}$ is $u_i$, for some $i$ with $3\leq i\leq n$. Then, $B^{'}$ is not empty and not a leaf. By Lemma \ref{lemma3p} and Remark \ref{remark4p}, we have $\tau(B,x)=1>\tau(B^{'},x)$. By Lemma \ref{minimal1}, we get that $\deg(u_1)\geq\deg(u_{i-1})$. Hence $\deg(u_1)\geq\max\{\deg(u_i):2\leq i\leq n-1\}$. Same reasoning leads to $\deg(u_n)\geq\max\{\deg(u_i):2\leq i\leq n-1\}$.
\end{proof}

\begin{defn}
Let $D=(d_1,d_2,\dots ,d_n)$ be a reduced degree sequence of a caterpillar for some $n\geq 2$. Let $k=\left\lceil\frac{n}{2}\right\rceil$. Then, $C_{L}^{k}(D)$ and $C_{R}^{k}(D)$ are defined as two disjoint complete branches as follows:
\begin{itemize}
\item[i)] If $n=2$, then $C_{L}^{k}(d_1,d_2)=C_{L}^{1}(d_1,d_2)=C(d_1-1)$ and $C_{R}^{k}(d_1,d_2)=C_{R}^{1}(d_1,d_2)=C(d_2-1)$.
\item[ii)] If $n=3$, then $C_{L}^{k}(d_1,d_2,d_3)=C_{L}^{2}(d_1,d_2,d_3)=C(d_1,d_3-1)$ and $C_{R}^{k}(d_1,d_2,d_3)=C_{R}^{2}(d_1,d_2,d_3)=C(d_2-1)$.
\item[iii)] If $n=4$, then $C_{L}^{k}(d_1,d_2,d_3,d_4)=C_{L}^{2}(d_1,d_2,d_3,d_4)=C(d_1,d_4-1)$ and $C_{R}^{k}(d_1,d_2,d_3,d_4)=C_{R}^{2}(d_1,d_2,d_3,d_4)=C(d_2,d_3-1)$.
\item[iv)] Suppose that $n\geq 5$. Let $n=4\ell+i$ for some integers $\ell$ and $i\in\{0,1,2,3\}$. In this case $k=2l+\lceil i/2 \rceil$. Then 
\begin{align*}
C_{L}^{k}(d_1,d_2,\dots ,d_n)
&=\begin{cases}C(d_1,d_n,d_3,d_{n-2},\dots ,d_{k+2},d_{k}-1),&\text{if $i=1$}\\
C(d_1,d_n,d_3,d_{n-2},\dots ,d_{k-1},d_{k+1}-1),&\text{if $i=3$}\\
C(d_1,d_n,d_3,d_{n-2},\dots ,d_{k+3},d_{k}-1),&\text{if $i=2$}\\
C(d_1,d_n,d_3,d_{n-2}\dots ,d_{k -1},d_{k+2}-1),&\text{if $i=0$}
\end{cases}
\end{align*}
and 
\begin{align*}
C_{R}^{k}(d_1,d_2,\dots ,d_n)
&=\begin{cases}C(d_2,d_{n-1},d_4,d_{n-3},\dots ,d_{k-1},d_{k+1}-1),&\text{if $i=1$}\\
C(d_2,d_{n-1},d_4,d_{n-3},\dots ,d_{k+2},d_{k}-1),&\text{if $i=3$}\\
C(d_2,d_{n-1},d_4,d_{n-3},\dots ,d_{k+2},d_{k+1}-1),&\text{if $i=2$}\\
C(d_2,d_{n-1},d_4,d_{n-3},\dots ,d_{k},d_{k+1}-1),&\text{if $i=0$.}\end{cases}
\end{align*}
\end{itemize}
\end{defn}

\begin{defn}
Let $D=(d_1,d_2,\dots ,d_n)$ be a reduced degree sequence of a caterpillar $G$. Then, we define $\mathcal{S}(D)$ to be the caterpillar of reduced degree sequence $D$, obtained by joining by an edge the root of $C_{L}^{k}(D)$ and the root of $C_{R}^{k}(D)$, where $k=\left\lceil\frac{n}{2}\right\rceil$. The roots of $C_{L}^{k}(D)$ and $C_{R}^{k}(D)$ are identified as follows. 
If $n=2$, then the root of $C_{L}^{k}(D)$ is the pseudo-leaf or a leaf of degree $d_1-1$, and the root of $C_{R}^{k}(D)$ is the pseudo-leaf or a leaf of degree $d_2-1$. If $n=3$, then the root of $C_{L}^{k}(D)$ is the pseudo-leaf or a leaf of degree $d_3-1$, and the root of $C_{R}^{k}(D)$ is the pseudo-leaf or a leaf of degree $d_2-1$. If $n=4$, then the root of $C_{L}^{k}(D)$ is the pseudo-leaf or a leaf of degree $d_4-1$, and the root of $C_{R}^{k}(D)$ is the pseudo-leaf or a leaf of degree $d_3-1$.

Note that for all integers $n\geq 5$, $C_{R}^{k}(d_3,\dots ,d_{n-2})$ is a subgraph of $C_{R}^{k}(d_1,d_2,\dots ,d_{n})$, and $C_{L}^{k}(d_3,\dots ,d_{n-2})$ is a subgraph of $C_{L}^{k}(d_1,d_2,\dots ,d_{n})$. To be precise,  $C_{R}^{k}(d_1,d_2,\dots ,d_{n})$ is obtained from $C_{R}^{k}(d_3,\dots ,d_{n-2})$  by attaching $d_{n-1}-1$ leaves to the leaf of $C_{R}^{k}(d_3,\dots ,d_{n-2})$ furthest from the root, and then attach $d_2-1$ more leaves to one of the newly attached leaves. After such a construction, we keep the root of $C_{R}^{k}(d_3,\dots ,d_{n-2})$ to be the root of $C_{R}^{k}(d_1,\dots ,d_{n})$. 
$C_{L}^{k}(d_1,\dots ,d_{n})$ is also an extension of $C_{L}^{k}(d_3,\dots ,d_{n-2})$ by iteratively attaching more leaves to a leaf further from the root. The root of  $C_{L}^{k}(d_3,\dots ,d_{n-2})$ is kept to the root of $C_{L}^{k}(d_1,\dots ,d_{n})$.

In this way, if $n=4\ell+i\geq 5$ for some integer $\ell$ and $0\leq i \leq 3$ we have
\begin{align*}
&\mathcal{S}(D)=\\
&\begin{cases}C(d_1,d_n,d_3,d_{n-2},\dots ,d_{\left\lceil\frac{n}{2}\right\rceil +2},d_{\left\lceil\frac{n}{2}\right\rceil},d_{\left\lceil\frac{n}{2}\right\rceil +1},d_{\left\lceil\frac{n}{2}\right\rceil -1},\dots ,d_{n-3},d_4,d_{n-1},d_2), \text{if  $i=1$}\\C(d_1,d_n,d_3,d_{n-2},\dots ,d_{\left\lceil\frac{n}{2}\right\rceil -1},d_{\left\lceil\frac{n}{2}\right\rceil +1},d_{\left\lceil\frac{n}{2}\right\rceil },d_{\left\lceil\frac{n}{2}\right\rceil +2},\dots ,d_{n-3},d_4,d_{n-1},d_2), \text{if $i=3$}\\C(d_1,d_n,d_3,d_{n-2},\dots ,d_{\left\lceil\frac{n}{2}\right\rceil +3},d_{\left\lceil\frac{n}{2}\right\rceil },d_{\left\lceil\frac{n}{2}\right\rceil +1},d_{\left\lceil\frac{n}{2}\right\rceil +2},\dots ,d_{n-3},d_4,d_{n-1},d_2), \text{if $i=2$}\\C(d_1,d_n,d_3,d_{n-2},\dots ,d_{\left\lceil\frac{n}{2}\right\rceil -1},d_{\left\lceil\frac{n}{2}\right\rceil +2},d_{\left\lceil\frac{n}{2}\right\rceil +1},d_{\left\lceil\frac{n}{2}\right\rceil},\dots ,d_{n-3},d_4,d_{n-1},d_2), \text{if $i=0$.}\end{cases}
\end{align*}
See Figure \ref{ILLUSTRATIONC1LK}, for a reduced degree sequence $\mathcal{S}(5,5,5,4,4,4,4,3,3,3)$.
\begin{figure}[htbp]
\centering
\includegraphics[scale=0.5]{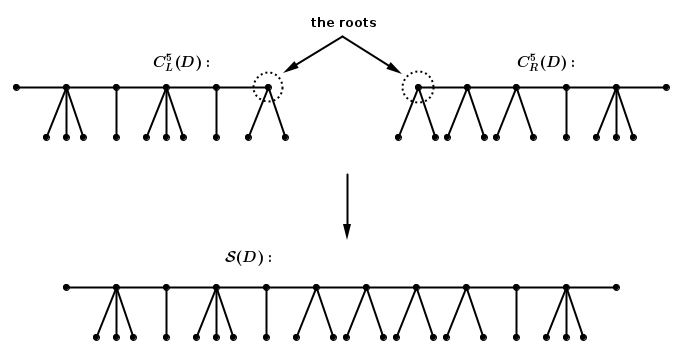}
\caption{The graphs $C_{L}^{5}(D)$, $C_{R}^{5}(D)$ and $\mathcal{S}(D)$, for $D=(5,5,5,4,4,4,4,3,3,3)$.\\}
\label{ILLUSTRATIONC1LK}
\end{figure}
\end{defn}

\begin{defn}
Let $G$ be a caterpillar and label all its non-leaf vertices from left to right as $u_1$, $u_2$,$\dots$,$u_n$. We define $G_L^{u_i}$ to be a complete branch of $G$ that contains $u_i$ but not $u_{i+1}$ and $G_R^{u_i}=G-G_L^{u_{i-1}}$ for $i\in \{1,\dots ,n-1\}$.
\end{defn}

\begin{thm}\label{minimal theorem}
Let $\mathbb{C}_D$ be the set of all caterpillars with reduced degree sequence $D$. Then, $M(\mathcal{S}(D),x)\leq M(H,x)$ for all $H\in\mathbb{C}_D$.
\end{thm}
\begin{proof}
Let $H$ be a caterpillar with reduced degree sequence $D=(d_1,d_2,\dots ,d_n)$. Suppose $H$ is minimal with respect to $M(.,x)$. Label all the non-leaf vertices in $H$ from left to right as $u_1,u_2,\dots ,u_n$.

i) If $n=1$, then $H$ is a star and $H=C(d_1)=C_{L}^{1}(d_1)$. Hence $H=\mathcal{S}(D)$, with $C_{R}^{1}(d_1)=()$.

ii) If $n=2$, then $H=C(d_1,d_2)$. $C(d_1,d_2)$ can be viewed as a caterpillar obtained by joining the roots of $C(d_1-1)=C_{L}^{1}(d_1,d_2)$ and $C(d_2-1)=C_{R}^{1}(d_1,d_2)$. Hence $H=C(d_1,d_2)=\mathcal{S}(D)$.

iii) If $n=3$, then by Lemma \ref{minimal2}, $u_1$ and $u_3$ attains the largest degrees in $H$. Assume $\deg(u_1)\geq\deg(u_3)$. Then, $H$ is the caterpillar of Figure \ref{FX1} (a), which can be viewed as in Figure \ref{FX1} (b), with $C_{L}^{2}(d_1,d_2,d_3)=C(d_1,d_3-1)$ and $C_{R}^{2}(d_1,d_2,d_3)=C(d_2-1)$.
\begin{figure}[htbp]
\centering
\includegraphics[scale=0.5]{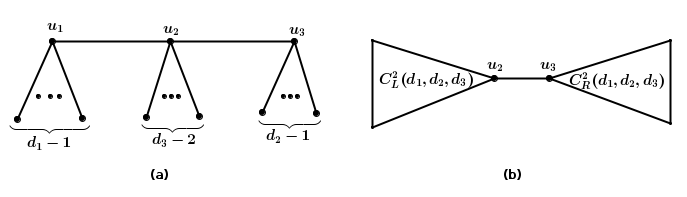}
\caption{The caterpillar $H$ in the proof of Theorem \ref{minimal theorem} for $n=3$.\\}
\label{FX1}
\end{figure}
Hence, $H\cong\mathcal{S}(D)$.

iv) If $n=4$, then by Lemma \ref{minimal2}, $u_1$ and $u_4$ must have the largest degrees in $H$. Assume $\deg(u_1)\geq\deg(u_4)$. Then, $$\tau(H_{L}^{u_1},x)=\frac{1}{1+x\left(\deg(u_1)-1\right)}\\\leq\frac{1}{1+x\left(\deg(u_4)-1\right)}=\tau(H_{R}^{u_4},x),$$ with equality if and only if $\deg(u_1)=\deg(u_4)$. If $\tau(H_{L}^{u_1},x)=\tau(H_{R}^{u_4},x)$, then $\deg(u_1)=\deg(u_4)$ and we can choose $\deg(u_2)\leq\deg(u_3)$. Otherwise $\tau(H_{L}^{u_1},x)<\tau(H_{R}^{u_4},x)$ and by Lemma \ref{minimal1}, we get that $\deg(u_2)\leq\deg(u_3)$. Hence, in all cases, $\deg(u_2)\leq\deg(u_3)$. Then, $H=C(d_1,d_4,d_3,d_2)$ is as in Figure \ref{FX2} (a), which can be viewed as in Figure \ref{FX2} (b).
\begin{figure}[htbp]
\centering
\includegraphics[scale=0.5]{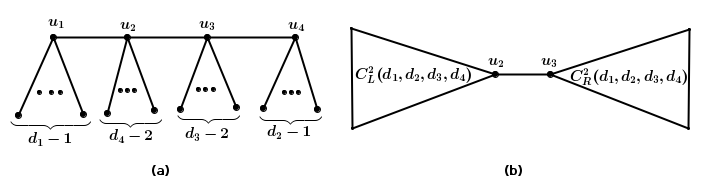}
\caption{The caterpillar $H$ in the proof of Theorem \ref{minimal theorem} for $n=4$.\\}
\label{FX2}
\end{figure} 
Hence $H\cong\mathcal{S}(D)$.

v) Suppose that $n\geq 5$. We are going to prove that $H$ satisfies the following characterisation of $\mathcal{S}(D)$. For $1\leq j< n$, if $j$ is odd, then $$\tau(H_{L}^{u_{j-1}},x)\geq\tau(H_{R}^{u_{n-j+2}},x)>\tau(H_{L}^{u_{i-2}},x)\text{, for all }j+2\leq i\leq n-j-1$$ and $$\deg(u_j)\geq\deg(u_{n-j+1})\geq\max\left\{\deg(u_{j+1}),\dots ,\deg(u_{n-j})\right\},$$ if $j$ is even then $$\tau(H_{L}^{u_{j-1}},x)\leq\tau(H_{R}^{u_{n-j+2}},x)<\tau(H_{L}^{u_{i-2}},x)\text{, for all }j+2\leq i\leq n-j-1$$and $$\deg(u_j)\leq\deg(u_{n-j+1})\leq\min\left\{\deg(u_{j+1}),\dots ,\deg(u_{n-j})\right\}.$$

Basis case: By Lemma \ref{minimal2}, $u_1$ and $u_n$ must have the largest degrees in $H$. Assume $\deg(u_1)\geq\deg(u_n)$. Therefore \begin{equation}\deg(u_1)\geq\deg(u_n)\geq\max\left\{\deg(u_2),\dots ,\deg(u_{n-1})\right\}.\label{onef}\end{equation}Then, \begin{equation}\tau(H_{L}^{u_1},x)=\frac{1}{1+x\left(\deg(u_1)-1\right)}\leq\frac{1}{1+x\left(\deg(u_n)-1\right)}=\tau(H_{R}^{u_n},x),\label{twof}\end{equation}with equality if and only if $\deg(u_1)=\deg(u_n)$. If $\tau(H_{L}^{u_1},x)=\tau(H_{R}^{u_n},x)$, we then choose $\deg(u_2)\leq\deg(u_{n-1})$. Otherwise, $\tau(H_{L}^{u_1},x)<\tau(H_{R}^{u_n},x)$ and by Lemma \ref{minimal1}, we get that $\deg(u_2)\leq\deg(u_{n-1})$. Hence, in all cases, \begin{equation}\deg(u_2)\leq\deg(u_{n-1}).\label{threef}\end{equation}From \eqref{onef}, we have $\deg(u_n)\geq\max\left\{\deg(u_i):2\leq i\leq n-1\right\}$. Since $n\geq 5$, then for $2\leq i\leq n-1$ the branch $H_{L}^{u_{i-1}}$ is neither empty nor a leaf. By Lemma \ref{lemma3p} and Remark \ref{remark4p}, we have $0<\tau(H_{L}^{u_{i-1}},x)<1$. Then, for $2\leq i\leq n-1$
\begin{align*}
\tau(H_{R}^{u_n},x)
&=\frac{1}{1+x\left(\deg(u_n)-1\right)}\\
&\leq\frac{1}{1+x\left(\deg(u_i)-1\right)},\hspace{2.2cm}\text{ since }\deg(u_n)\geq\deg(u_i)\\
&<\frac{1}{1+x\left(\deg(u_i)-2+\tau(H_{L}^{u_{i-1}},x)\right)},\text{ since }0<\tau(H_{L}^{u_{i-1}},x)<1\\
&=\tau(H_{L}^{u_i},x).
\end{align*}
Hence \begin{equation}\tau(H_{R}^{u_n},x)<\tau(H_{L}^{u_i},x)\text{, for all }2\leq i\leq n-1.\label{three.1}\end{equation}By Lemma \ref{minimal1}, \begin{equation}\deg(u_{n-1})\leq\deg(u_{i+1})\text{, for all }2\leq i\leq n-3.\label{fourf}\end{equation}From \eqref{threef} and \eqref{fourf} we get that $\deg(u_2)\leq\deg(u_{n-1})\leq\deg(u_{i+1})$, for all $2\leq i\leq n-3$. Therefore, \begin{equation}\deg(u_2)\leq\deg(u_{n-1})\leq\min\left\{\deg(u_3),\dots ,\deg(u_{n-2})\right\}.\label{fivef}\end{equation}From \eqref{twof} we have $\tau(H_{L}^{u_1},x)\leq\tau(H_{R}^{u_n},x)$, with equality if and only if $\deg(u_1)=\deg(u_n)$. From \eqref{fivef} we get that $\deg(u_2)\leq\deg(u_{n-1})$. Since $\deg(u_2)\leq\deg(u_{n-1})$ and $\tau(H_{L}^{u_1},x)\leq\tau(H_{R}^{u_n},x)$, then
\begin{align*}
\tau(H_{L}^{u_2},x)
&=\frac{1}{1+x\left(\deg(u_2)-2+\tau(H_{L}^{u_1},x)\right)}\\
&\geq\frac{1}{1+x\left(\deg(u_{n-1})-2+\tau(H_{R}^{u_n},x)\right)}
=\tau(H_{R}^{u_{n-1}},x).
\end{align*}
Therefore, \begin{equation}\tau(H_{L}^{u_2},x)\geq\tau(H_{R}^{u_{n-1}},x).\label{five.1}\end{equation}By Lemma \ref{lemma4m}, we get $\tau(H_{L}^{u_2},x)=\tau(H_{R}^{u_{n-1}},x)$ if and only if $H_{L}^{u_2}\approx_r H_{R}^{u_{n-1}}$, that is, if $\deg(u_2)=\deg(u_{n-1})$ and $\deg(u_1)=\deg(u_n)$. If $\tau(H_{L}^{u_2},x)=\tau(H_{R}^{u_{n-1}},x)$, then choose $\deg(u_3)\geq\deg(u_{n-2})$. Otherwise $\tau(H_{L}^{u_2},x)>\tau(H_{R}^{u_{n-1}},x)$ and by Lemma \ref{minimal1}, we get $\deg(u_3)\geq\deg(u_{n-2})$. Hence \begin{equation}\deg(u_3)\geq\deg(u_{n-2}).\label{sixf}\end{equation}For any $3\leq i\leq n-2$, $H_{L}^{u_{i-2}}$ is neither empty nor a leaf. Then, by Lemma \ref{lemma3p} and Remark \ref{remark4p}, we have $0<\tau(H_{L}^{u_{i-2}},x)<1$. And for any $3\leq i\leq n-2$, we get
\begin{align*}
\tau(H_{R}^{u_n},x)
&=\frac{1}{1+x\left(\deg(u_n)-1\right)}\\
&\leq\frac{1}{1+x\left(\deg(u_{i-1})-1\right)}\text{, since }\deg(u_n)\geq\deg(u_{i-1})\text{ in }\eqref{onef}\\
&<\frac{1}{1+x\left(\deg(u_{i-1})-2+\tau(H_{L}^{u_{i-2}},x)\right)}\text{, since }0<\tau(H_{L}^{u_{i-2}},x)<1\\
&=\tau(H_{L}^{u_{i-1}},x).
\end{align*}
Therefore, \begin{equation}\tau(H_{R}^{u_{n}},x)<\tau(H_{L}^{u_{i-1}},x)\text{, for any }3\leq i\leq n-2.\label{sevenf}\end{equation}From \eqref{fivef} and \eqref{sevenf} we get 
\begin{align*}
\tau(H_{R}^{u_{n-1}},x)
&=\frac{1}{1+x\left(\deg(u_{n-1})-2+\tau(H_{R}^{u_n},x)\right)}
>\frac{1}{1+x\left(\deg(u_i)-2+\tau(H_{L}^{u_{i-1}},x)\right)}\\
&=\tau(H_{L}^{u_i},x).
\end{align*} 
Therefore, \begin{equation}\tau(H_{R}^{u_{n-1}},x)>\tau(H_{L}^{u_i},x)\text{, for any }3\leq i\leq n-2.\label{seven.1}\end{equation}By Lemma \ref{minimal1}, we get that $\deg(u_{n-2})\geq\deg(u_{i+1})$ for any $3\leq i\leq n-4$. Hence \begin{equation}\deg(u_{n-2})\geq\max\left\{\deg(u_4),\dots ,\deg(u_{n-3})\right\}.\label{eightf}\end{equation}From \eqref{sixf} and \eqref{eightf} we get that \begin{equation}\deg(u_3)\geq\deg(u_{n-2})\geq\max\left\{\deg(u_4),\dots ,\deg(u_{n-3})\right\}.\label{ninef}\end{equation}Therefore, from \eqref{onef} we have $$\deg(u_1)\geq\deg(u_n)\geq\max\left\{\deg(u_2),\dots ,\deg(u_{n-1})\right\},$$from \eqref{twof} and \eqref{three.1} we have $$\tau(H_{L}^{u_1},x)\leq\tau(H_{R}^{u_n},x)<\tau(H_{L}^{u_{i-2}},x)\text{, for any }4\leq i\leq n-3$$and from \eqref{fivef} we have $$\deg(u_2)\leq\deg(u_{n-1})\leq\min\left\{\deg(u_3),\dots ,\deg(u_{n-2})\right\}.$$From \eqref{five.1} and \eqref{seven.1} we have $$\tau(H_{L}^{u_2},x)\geq\tau(H_{R}^{u_{n-1}},x)>\tau(H_{L}^{u_{i-2}},x)\text{, for any }5\leq i\leq n-4$$ and from \eqref{ninef} we have $$\deg(u_3)\geq\deg(u_{n-2})\geq\max\left\{\deg(u_4),\dots ,\deg(u_{n-3})\right\}.$$Suppose that for all integers $1\leq j<n$:\\(a) If $j$ is odd, then \begin{equation}
\tau(H_{L}^{u_{j-1}},x)\geq\tau(H_{R}^{u_{n-j+2}},x)>\tau(H_{L}^{u_{i-2}},x)\text{, for any }j+2\leq i\leq n-j-1,\label{tenf}
\end{equation}and \begin{equation}
\deg(u_j)\geq\deg(u_{n-j+1})\geq\max\left\{\deg(u_{j+1}),\dots ,\deg(u_{n-j})\right\}.\label{elevenf}
\end{equation}(b) If $j$ is even, then \begin{equation}
\tau(H_{L}^{u_{j-1}},x)\leq\tau(H_{R}^{u_{n-j+2}},x)<\tau(H_{L}^{u_{i-2}},x)\text{, for any }j+2\leq i\leq n-j-1\label{twelvef}
\end{equation}and \begin{equation}
\deg(u_j)\leq\deg(u_{n-j+1})\leq\min\left\{\deg(u_{j+1}),\dots ,\deg(u_{n-j})\right\}.\label{thirteenf}
\end{equation}(i) Suppose $j$ is odd. Then, from \eqref{tenf} and \eqref{elevenf} we have $\tau(H_{L}^{u_{j-1}},x)\geq\tau(H_{R}^{u_{n-j+2}},x)$ and $\deg(u_j)\geq\deg(u_{n-j+1})$. Hence,
\begin{align*}
\tau(H_{L}^{u_j},x)
&=\frac{1}{1+x\left(\deg(u_j)-2+\tau(H_{L}^{u_{j-1}},x)\right)}\\
&\leq\frac{1}{1+x\left(\deg(u_{n-j+1})-2+\tau(H_{R}^{u_{n-j+2}},x)\right)}
=\tau(H_{R}^{u_{n-j+1}},x).
\end{align*}
Therefore, \begin{equation}
\tau(H_{L}^{u_j},x)\leq \tau(H_{R}^{u_{n-j+1}},x).\label{fourteenf}
\end{equation}
By Lemma \ref{lemma4m}, $\tau(H_{L}^{u_j},x)=\tau(H_{R}^{u_{n-j+1}},x)$ if and only if $H_{L}^{u_{j}}\approx_r H_{R}^{u_{n-j+1}}$, in such a case $\deg(u_j)=\deg(u_{n-j+1})$ and $\tau(H_{L}^{u_{j-1}},x)=\tau(H_{R}^{u_{n-j+2}},x).$ Otherwise, $\tau(H_{L}^{u_j},x)<\tau(H_{R}^{u_{n-j+1}},x)$. If $\tau(H_{L}^{u_j},x)=\tau(H_{R}^{u_{n-j+1}},x)$, then we choose $\deg(u_{j+1})\leq\deg(u_{n-j})$. Otherwise $\tau(H_{L}^{u_j},x)<\tau(H_{R}^{u_{n-j+1}},x)$ and by Lemma \ref{minimal1}, we must have $\deg(u_{j+1})\leq\deg(u_{n-j})$. Hence, we laways have \begin{equation}
\deg(u_{j+1})\leq\deg(u_{n-j}).\label{fifteenf}
\end{equation}From \eqref{tenf}, we have $\tau(H_{R}^{u_{n-j+2}},x)>\tau(H_{L}^{u_{i-2}},x)$, for any $j+2\leq i\leq n-j-1$. Then, by Lemma \ref{minimal1}, we must have $\deg(u_{n-j+1})\geq\deg(u_{i-1})$ and hence,
\begin{align*}
\tau(H_{R}^{u_{n-j+1}},x)
&=\frac{1}{1+x\left(\deg(u_{n-j+1})-2+\tau(H_{R}^{u_{n-j+2}},x)\right)}\\
&<\frac{1}{1+x\left(\deg(u_{i-1})-2+\tau(H_{L}^{u_{i-2}},x)\right)}
=\tau(H_{L}^{u_{i-1}},x).
\end{align*}
Therefore, \begin{equation}
\tau(H_{R}^{u_{n-j+1}},x)<\tau(H_{L}^{u_{i-1}},x)\text{, for any }j+2\leq i\leq n-j-1.\label{sixteenf}
\end{equation}Hence, from \eqref{fourteenf} and \eqref{sixteenf} we get  \begin{equation}
\tau(H_{L}^{u_j},x)\leq\tau(H_{R}^{u_{n-j+1}},x)<\tau(H_{L}^{u_{i-1}},x)\text{, for any }j+2\leq i\leq n-j-1.\label{seventeenf}
\end{equation}From \eqref{seventeenf} we have $\tau(H_{R}^{u_{n-j+1}},x)<\tau(H_{L}^{u_{i-1}},x)$. By Lemma \ref{minimal1}, we must have $\deg(u_{n-j})\leq\deg(u_i)$. Hence \begin{equation}
\deg(u_{n-j})\leq\min\left\{\deg(u_{j+2}),\dots ,\deg(u_{n-j-1})\right\}.\label{eighteenf}
\end{equation}Therefore, from \eqref{fifteenf} and \eqref{eighteenf} we must have \begin{equation}
\deg(u_{j+1})\leq\deg(u_{n-j})\leq\min\left\{\deg(u_{j+2}),\dots ,\deg(u_{n-j-1})\right\}.\label{ninteenf}
\end{equation}Hence, from \eqref{seventeenf} and \eqref{ninteenf} if $j$ is odd, then $$\tau(H_{L}^{u_j},x)\leq\tau(H_{R}^{u_{n-j+1}},x)<\tau(H_{L}^{u_{i-1}},x)\text{, for any }j+2\leq i\leq n-j-1$$ and $$\deg(u_{j+1})\leq\deg(u_{n-j})\leq\min\left\{\deg(u_{j+2}),\dots ,\deg(u_{n-j-1})\right\}.$$ ii) Suppose $j$ is even. From \eqref{twelvef} and \eqref{thirteenf} we have $\tau(H_{L}^{u_{j-1}},x)\leq\tau(H_{R}^{u_{n-j+2}},x)$ and $\deg(u_{j})\leq\deg(u_{n-j+1})$. Then,
\begin{align*}
\tau(H_{L}^{u_j},x)
&=\frac{1}{1+x\left(\deg(u_j)-2+\tau(H_{L}^{u_{j-1}},x)\right)}\\
&\geq\frac{1}{1+x\left(\deg(u_{n-j+1})-2+\tau(H_{R}^{u_{n-j+2}},x)\right)}
=\tau(H_{R}^{u_{n-j+1}},x).
\end{align*}
Therefore, \begin{equation}
\tau(H_{L}^{u_j},x)\geq\tau(H_{R}^{u_{n-j+1}},x).\label{twentyf}
\end{equation}By Lemma \ref{lemma4m}, $\tau(H_{L}^{u_j},x)=\tau(H_{R}^{u_{n-j+1}},x)$ if and only if $H_{L}^{u_j}\approx_r H_{R}^{u_{n-j+1}}$, in such a case $\deg(u_j)=\deg(u_{n-j+1})$ and $\tau(H_{L}^{u_{j-1}},x)=\tau(H_{R}^{u_{n-j+2}},x)$. Otherwise $\tau(H_{L}^{u_j},x)>\tau(H_{R}^{u_{n-j+1}},x)$. If $\tau(H_{L}^{u_j},x)=\tau(H_{R}^{u_{n-j+1}},x)$, then we choose $\deg(u_{j+1})\geq\deg(u_{n-j})$. Otherwise, $\tau(H_{L}^{u_j},x)>\tau(H_{R}^{u_{n-j+1}},x)$ and by Lemma \ref{minimal1}, we must have $\deg(u_{j+1})\geq\deg(u_{n-j})$. Hence, in all cases we have \begin{equation}
\deg(u_{j+1})\geq\deg(u_{n-j}).\label{twentyonef}
\end{equation}From \eqref{twelvef} we have $\tau(H_{R}^{u_{n-j+2}},x)<\tau(H_{L}^{u_{i-2}},x)$, for any $j+2\leq i\leq n-j-1$. By Lemma \ref{minimal1}, we must have $\deg(u_{n-j+1})\leq\deg(u_{i-1})$ and hence
\begin{align*}
\tau(H_{R}^{u_{n-j+1}},x)
&=\frac{1}{1+x\left(\deg(u_{n-j+1})-2+\tau(H_{R}^{u_{n-j+2}},x)\right)}\\
&>\frac{1}{1+x\left(\deg(u_{i-1})-2+\tau(H_{L}^{u_{i-2}},x)\right)}
=\tau(H_{L}^{u_{i-1}},x).
\end{align*}
Therefore, \begin{equation}
\tau(H_{R}^{u_{n-j+1}},x)>\tau(H_{L}^{u_{i-1}},x)\text{, for any }j+2\leq i\leq n-j-1.\label{twentytwof}
\end{equation}Hence, from \eqref{twentyf} and \eqref{twentytwof} we must have \begin{equation}
\tau(H_{L}^{u_j},x)\geq\tau(H_{R}^{u_{n-j+1}},x)>\tau(H_{L}^{u_{i-1}},x)\text{, for any }j+2\leq i\leq n-j-1.\label{twentythreef}
\end{equation}From \eqref{twentythreef} we have $\tau(H_{R}^{u_{n-j+1}},x)>\tau(H_{L}^{u_{i-1}},x)$, for any $j+2\leq i\leq n-j-1$. By Lemma \ref{minimal1}, we must have $\deg(u_{n-j})\geq\deg(u_{i})$. Hence \begin{equation}
\deg(u_{n-j})\geq\max\left\{\deg(u_{j+2}),\dots ,\deg(u_{n-j-1})\right\}.\label{twentyfourf}
\end{equation}From \eqref{twentyonef} and \eqref{twentyfourf} we have \begin{equation}
\deg(u_{j+1})\geq\deg(u_{n-j})\geq\max\left\{\deg(u_{j+2}),\dots ,\deg(u_{n-j-1})\right\}.\label{twentyfivef}
\end{equation}Hence, from \eqref{twentythreef} and \eqref{twentyfivef} if $j$ is even, then $$\tau(H_{L}^{u_j},x)\geq\tau(H_{R}^{u_{n-j+1}},x)>\tau(H_{L}^{u_{i-1}},x)\text{, for all }j+2\leq i\leq n-j-1$$ and $$\deg(u_{j+1})\geq\deg(u_{n-j})\geq\max\left\{\deg(u_{j+2}),\dots ,\deg(u_{n-j-1})\right\}.$$ 
\end{proof}

\begin{rem}
\label{Rem:MTOZE}
In section 1, we saw that if $T$ and $T^{'}$ are trees, such that $M(T,x)\leq M(T^{'},x)$ for all positive $x\in\R$, then $$Z(T)=M(T,1)\leq M(T^{'},1)=Z(T^{'}),$$and $$En(T)=\frac{2}{\pi}\int_{0}^{+\infty}\frac{1}{x^2}\ln M(T,x^2)dx\leq\frac{2}{\pi}\int_{0}^{+\infty}\frac{1}{x^2}\ln M(T^{'},x^{2})dx=En(T^{'}).$$
\end{rem}
Hence, from Theorem \ref{minimal theorem} we deduce the following corollary.
\begin{cor}\label{minimal theorem Z and En}
Let $\mathbb{C}_D$ be the set of all caterpillars with reduced degree sequence $D$. Then, $Z(\mathcal{S}(D))\leq Z(H)$ and $En(\mathcal{S}(D))\leq En(H)$ for all $H\in\mathbb{C}_D$.
\end{cor}

\section{Caterpillars with different degree sequences}
\label{Sec:DiffDeg}

\begin{defn}
Let $(b_1,\dots ,b_n)$ and $(d_1,\dots ,d_n)$ be two degree sequences of trees. We say that $(d_1,\dots ,d_n)$ majorizes $(b_1,\dots ,b_n)$ and write $(b_1,\dots ,b_n)\preceq (d_1,\dots ,d_n)$, if and only if for all $k\in\{1,\dots ,n\}$ we have $$\sum_{i=1}^{k}b_i\leq\sum_{i=1}^{k}d_i.$$If furthermore there exists $i_0$, such that $d_{i_0}\neq b_{i_0}$, then we write $(b_1,\dots ,b_n)\prec (d_1,\dots ,d_n)$.
\end{defn}

This section compares $\mathcal{S}(D)$ and $\mathcal{S}(D')$ when $D'$ majorises $D$. Combined with Theorem \ref{minimal theorem} and Corollary \ref{minimal theorem Z and En}, this leads to a few corollaries on classes of trees with various degree conditions.
First, we consider the case where there are exactly two positions where the entries in $D$ and $D'$ are different. The following lemma describes a transfer of a leaf, from a small degree vertex of $\mathcal{S}_D$ to one with a larger degree, that decreases $M(.,x).$ 
\begin{lem}\label{minimal3}
Let $i,j$ and $n$ be positive integers such that $1\leq i<j\leq n$. Let $D=(d_1,\dots ,d_i,\dots ,d_j,\dots ,d_n)$ be reduced degree sequence of a caterpillar $C$, such that $d_i\geq d_j>2$. Decompose $\mathcal{S}(D)$ as in Figure \ref{FX3}, with $B$, $B^{'}$ and $H$ non-empty. Let $v$ and $w$ be vertices of  $\mathcal{S}(D)$ such that $\deg(v)=d_i$ and $\deg(w)=d_j$. Let $w'$ be a leaf adjacent to $w$. Let $G$ be obtained from $\mathcal{S}(D)$ by removing the edge $ww'$ and then adding the edge $vw'$. If either i) $d_i>d_j$ or ii) $d_i=d_j$ and $M(H-w,x)\geq M(H-v,x)$, then $$M(\mathcal{S}(D),x)>M(G,x)\geq M(\mathcal{S}(D^{'}),x),$$ where $D^{'}=(d_1,\dots ,d_{i-1},d_i+1,d_{i+1},\dots ,d_{j-1},d_{j}-1,d_{j+1},\dots ,d_n)$.
\begin{figure}[htbp]
\centering
\includegraphics[scale=0.5]{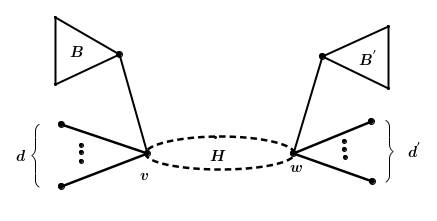}
\caption{Decomposition of the caterpillar $\mathcal{S}(D)$ in Lemma \ref{minimal3}.}
\label{FX3}
\end{figure}
\end{lem}
\begin{proof}
Let $d=d_i-2$ and $d^{'}=d_j-2$. Suppose $\deg(v)=d_i>d_j=\deg(w)$. Then, $d>d^{'}$. Since $\mathcal{S}(D)$ is minimal with respect to $M(.,x)$ and $d>d^{'}$, then by Lemma \ref{minimal1} ii) and iii), we must have $\tau(B,x)\geq\tau(B^{'},x)$ and $M(H-v,x)\leq M(H-w,x)$. From Equation \eqref{M(G,x)}, we have
\begin{align*}
M(\mathcal{S}(D),x)
&=M(B,x)M(B^{'},x)\left[M(H,x)+dd^{'}x^2M(H-\{v,w\},x)\right.\\
&\hspace*{2cm}+\left.x^2\tau(B,x)\tau(B^{'},x)M(H-\{v,w\},x)+M_{v}^{w}(\mathcal{S}(D),x)\right]
\end{align*}
where,
\begin{align*}
M_{v}^{w}(\mathcal{S}(D),x)
&=x\left[\left(d^{'}+\tau(B^{'},x)\right)M(H-w,x)+\left(d+\tau(B,x)\right)M(H-v,x)\right]\\
&+x^2M(H-\{v,w\},x)\left[d\tau(B^{'},x)+d^{'}\tau(B,x)\right]
\end{align*}
and
\begin{align*}
M(G,x)
&=M(B,x)M(B^{'},x)\left[M(H,x)+(d+1)(d^{'}-1)x^2M(H-\{v,w\},x)\right.\\
&\hspace*{3.2cm}+\left.x^2\tau(B,x)\tau(B^{'},x)M(H-\{v,w\},x)+M_{v}^{w}(G,x)\right]
\end{align*}
where,
\begin{align*}
M_{v}^{w}(G,x)
&=x\left[\left(d^{'}+\tau(B^{'},x)-1\right)M(H-w,x)+(d+\tau(B,x)+1)M(H-v,x)\right]\\
&\hspace*{2.3cm}+x^2M(H-\{v,w\},x)\left[(d+1)\tau(B^{'},x)+\left(d^{'}-1\right)\tau(B,x)\right].
\end{align*}
Then,
\begin{align*}
&M(\mathcal{S}(D),x)-M(G,x)\\
&=M(B,x)M(B^{'},x)\left[\left(d-d^{'}+1\right)x^2M(H-\{v,w\},x)+M_{v}^{w}(\mathcal{S}(D),x)-M_{v}^{w}(G,x)\right].
\end{align*}
But,
\begin{align*}
&M_{v}^{w}(\mathcal{S}(D),x)-M_{v}^{w}(G,x)\\
&=x\left[M(H-w,x)-M(H-v,x)\right]+x^2M(H-\{v,w\},x)\left[\tau(B,x)-\tau(B^{'},x)\right].
\end{align*}
Therefore,
\begin{align*}
&M(\mathcal{S}(D),x)-M(G,x)\\
&=M(B,x)M(B^{'},x)\left[x^2M(H-\{v,w\},x)\left[(d+\tau(B,x)+1)-\left(d^{'}+\tau(B^{'},x)\right)\right]\right.\\
&\hspace{7.15cm}+x\left.\left(M(H-w,x)-M(H-v,x)\right)\right]\\
&>0\text{, since }x>0,0<\tau(B,x),\tau(B^{'},x)\leq 1,d>d^{'}\text{ and }M(H-w,x)\geq M(H-v,x).
\end{align*}

Suppose $d_i=d_j$ and $M(H-w,x)\geq M(H-v,x)$. Then, $d=d^{'}$ and 
\begin{align*}
&M(\mathcal{S}(D),x)-M(G,x)\\
&=M(B,x)M(B^{'},x)\left[x^2M(H-\{v,w\},x)\left(\tau(B,x)+1-\tau(B^{'},x)\right)\right.\\
&\hspace*{7.55cm}+x\left.\left(M(H-w,x)-M(H-v,x)\right)\right]\\
&>0\text{, since }x>0,0<\tau(B,x),\tau(B^{'},x)\leq 1\text{ and }M(H-w,x)\geq M(H-v,x).
\end{align*}
Hence $M(\mathcal{S}(D),x)>M(G,x)$. Since $\mathcal{S}(D^{'})$ and $G$ have the same reduced degree sequence $D^{'}$ and are both caterpillars, then by Theorem \ref{minimal theorem}, we must have $M(G,x)\geq M(\mathcal{S}(D^{'}),x)$. Hence, we conclude $$M(\mathcal{S}(D),x)>M(G,x)\geq M(\mathcal{S}(D^{'}),x).$$
\end{proof}
Then next lemma is comparable to Lemma \ref{minimal3}, except that the transformed graph is not necessarily $\mathcal{S}(D)$.
\begin{lem}\label{lemma needed}
Let $G$ be a caterpillar and be decomposed as in Figure \ref{Figure-1}, with $B$, $B^{'}$ and $H$ non-empty. Let $w^{'}$ be a leaf adjacent to $w$. Let $G^{'}$ be obtained from $G$ by removing the edge $ww{'}$ and then adding the edge $vw^{'}$. If $\tau(B,x)\geq\tau(B^{'},x)$, $d\geq d^{'}$ and $M(H-v,x)\leq M(H-w,x)$, then $M(G,x)>M(G^{'},x)$.
\end{lem}
\begin{proof}
Suppose $B$, $B^{'}$ and $H$ are non-empty. Suppose $\tau(B,x)\geq\tau(B^{'},x)$, $d\geq d^{'}$ and $M(H-v,x)\leq M(H-w,x)$.
\begin{align*}
M(G,x)
&=M(B,x)M(B^{'},x)\left[M(H,x)+dd^{'}x^2M(H-\{v,w\},x)\right.\\
&\hspace*{3cm}+\left.x^2\tau(B,x)\tau(B^{'},x)M(H-\{v,w\},x)+M_{v}^{w}(G,x)\right]
\end{align*}
and
\begin{align*}
M(G^{'},x)
&=M(B,x)M(B^{'},x)\left[M(H,x)+(d+1)\left(d^{'}-1\right)x^2M(H-\{v,w\},x)\right.\\
&\hspace*{3cm}+\left.x^2\tau(B,x)\tau(B^{'},x)M(H-\{v,w\},x)+M_{v}^{w}(G^{'},x)\right].
\end{align*}
Then,
\begin{align*}
&M(G,x)-M(G^{'},x)\\
&=M(B,x)M(B^{'},x)\left[x^2M(H-\{v,w\},x)\left(d-d^{'}+1\right)+M_{v}^{w}(G,x)-M_{v}^{w}(G^{'},x)\right].
\end{align*}
Using
\begin{align*}
&M_{v}^{w}(G,x)-M_{v}^{w}(G^{'},x)\\
&=x\left[M(H-w,x)-M(H-v,x)\right]+x^2M(H-\{v,w\},x)\left[\tau(B,x)-\tau(B^{'},x)\right],
\end{align*}
we have
\begin{align*}
&M(G,x)-M(G^{'},x)\\
&=M(B,x)M(B^{'},x)\left[x^2M(H-\{v,w\},x)\left(d+\tau(B,x)+1-\left(d^{'}+\tau(B^{'},x)\right)\right)\right.\\
&\hspace{7.5cm}+\left.x\left[M(H-w,x)-M(H-v,x)\right]\right]\\
&>0\text{, since }d\geq d^{'},\tau(B,x)\geq\tau(B^{'},x)\text{ and }M(H-w,x)\geq M(H-v,x).
\end{align*}
\end{proof}

Now, we describe a situation when a transfer of a possibly big branch decreases $M(.,x)$.
\begin{lem}\label{minimal4}
Let $G$ be a caterpillar and be decomposed  as in Figure \ref{Figure-1}, with $B$, $B^{'}$ and $H$ non-empty. Let $G^{'}$ be obtained from $G$ by removing the edge $wr(B^{'})$ and then adding the edge $vr(B^{'})$. Suppose $G$ is minimal with respect to $M(.,x)$. If either
\begin{itemize}
\item[1)] $\tau(B,x)>\tau(B^{'},x)$,

\item[2)] $d>d^{'}$, or

\item[3)] $\tau(B,x)=\tau(B^{'},x)$, $d=d^{'}$ and $M(H-v,x)\leq M(H-w,x)$, 
\end{itemize}
then $$M(G,x)>M(G^{'},x).$$
\end{lem}
\begin{proof}
Suppose $G$ is a caterpillar and decomposed as in Figure \ref{Figure-1}. Suppose $G$ is minimal with respect to $M(.,x)$. Suppose $\tau(B,x)>\tau(B^{'},x)$, then by Lemma \ref{minimal1} i), we must have $d\geq d^{'}$ and $M(H-v,x)\leq M(H-w,x)$. Suppose $d>d^{'}$, then by Lemma \ref{minimal1} ii) and iii), we must have $M(H-v,x)\leq M(H-w,x)$. Therefore, if 1), 2) or 3) holds, then $d\geq d^{'}$ and $M(H-v,x)\leq M(H-w,x)$. $G^{'}$ can be decomposed as in Figure \ref{FX4}.
\begin{figure}[htbp]
\centering
\includegraphics[scale=0.5]{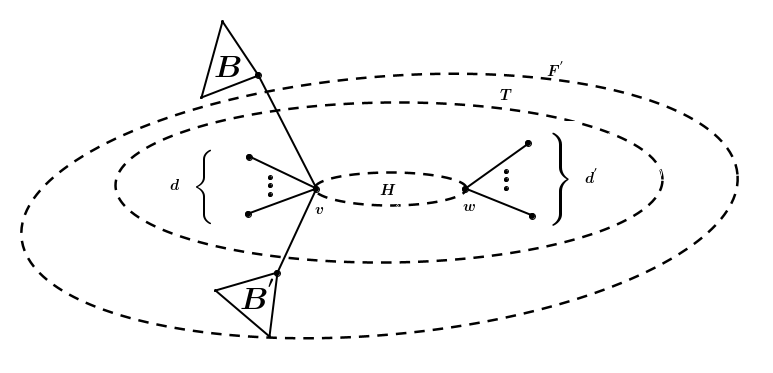}
\caption{Decomposition of the tree $G^{'}$ in the proof of Lemma \ref{minimal4}. }
\label{FX4}
\end{figure}
Hence, we have
\begin{align*}
M(G^{'},x)
&=M(B,x)M(B^{'},x)M(T,x)+xM(B,x)M(B^{'}-r(B^{'}),x)M(T-v,x)\\
&\hspace{5.3cm}+xM(B-r(B),x)M(B^{'},x)M(T-v,x)
\end{align*}
and
\begin{align*}
M(G,x)
&=M(B,x)M(B^{'},x)M(T,x)+xM(B,x)M(B^{'}-r(B^{'}),x)M(T-w,x)\\
&\hspace{5.3cm}+xM(B^{'},x)M(B-r(B),x)M(T-v,x)\\
&+x^2M(B-r(B),x)M(B^{'}-r(B^{'}),x)M(T-\{v,w\},x).
\end{align*}
Then,
\begin{align*}
&M(G,x)-M(G^{'},x)\\
&=xM(B,x)M(B^{'}-r(B^{'}),x)\left[M(H-w,x)-M(H-v,x)\right.\\
&\hspace*{7.25cm}\left.+xM(H-\{v,w\},x)(d-d^{'})\right]\\
&+x^2M(B-r(B),x)M(B^{'}-r(B^{'}),x)M(H-\{v,w\},x)
>0,
\end{align*}
$\text{since }x>0,M(B-r(B),x),M(B^{'}-r(B^{'}),x),M(H-\{v,w\},x)>0,M(H-w,x)\geq M(H-v,x)\text{ and }d\geq d^{'}.$
\end{proof}

\begin{thm}\label{minimal theorem2}
Let $(y_1,\dots ,y_n)$ and $(d_1,\dots ,d_n)$ be two degree sequences of caterpillars. If $(y_1,\dots ,y_n)\prec (d_1,\dots ,d_n)$ and $\displaystyle{\sum_{i=1}^{n}y_i=\sum_{i=1}^{n}d_i}$, then for all $x>0$ we have $$M(\mathcal{S}(y_1,\dots ,y_n),x)>M(\mathcal{S}(d_1,\dots ,d_n),x)$$and hence,$$Z(\mathcal{S}(y_1,\dots ,y_n))>Z(\mathcal{S}(d_1,\dots ,d_n))$$and $$En(\mathcal{S}(y_1,\dots ,y_n))>En(\mathcal{S}(d_1,\dots ,d_n)).$$
\end{thm}
\begin{proof}
Suppose ${Y}=(y_1,\dots ,y_n)$ and $D=(d_1,\dots ,d_n)$ are two degree sequences of caterpillars, ${Y}\prec D$ and $\displaystyle{\sum_{i=1}^{n}y_i=\sum_{i=1}^{n}d_i}$. Then, there exists $i_0$, such that $d_{i_0}\neq y_{i_0}$. Infact, the set $\mathbb{I}=\{i:d_i\neq y_i\}$ must have  at least two elements, otherwise $\displaystyle{\sum_{i=1}^{n}y_i=\sum_{i=1}^{n}d_i}$ would be impossible. Let $l=\min\{i:d_i\neq y_i\}$ and $m=\max\{i:d_i\neq y_i\}$. We must have $y_m>d_m\geq 1$ and we must also have $d_l>y_l$. We define $${Y}_1=(y_1,\dots ,y_{l-1},y_{l}+1,y_{l+1},\dots ,y_{m-1},y_{m}-1,y_{m+1},\dots ,y_n).$$Note that ${Y}_1$ is still a valid degree sequence. If $l>1$, then $y_{l-1}=d_{l-1}\geq d_l\geq y_{l}+1> d_m+1\geq 2$ and if $l=1$, then $d_l\geq y_l+1>d_m+1\geq 2$. If $m<n$, then $y_m-1\geq d_m\geq d_{m+1}=y_{m+1}\geq 1$ and if $m=n$, then $y_m-1\geq d_m\geq 1$. It is clear that ${Y}\prec {Y}_1$. If $y_m>2$, then by applying Lemma \ref{minimal3} to $\mathcal{S}({Y})$, we know that there exists a caterpillar $G$ with degree sequence ${Y}_1$ such that $$M(\mathcal{S}({Y}),x)>M(G,x)\geq M(\mathcal{S}({Y}_1),x).$$Otherwise $y_m=2$. In such case, decompose $\mathcal{S}({Y})$ as in Figure \ref{figuresigma}, with $B$ non-empty, $\deg(v)=y_l$ and $\deg(w)=y_m=2$.
\begin{figure}[htbp]
\centering
\includegraphics[scale=0.5]{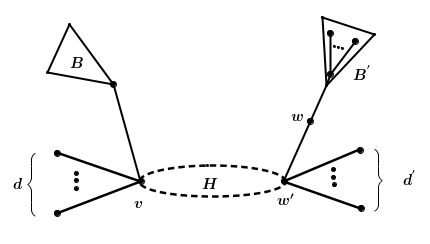}
\caption{Decomposition of the caterpillar $\mathcal{S}({Y})$ in the proof of Theorem \ref{minimal theorem2}.}
\label{figuresigma}
\end{figure} 
In view of the structure of $\mathcal{S}(Y)$, we can choose $w$ so that $\deg(r(B^{'}))\geq\deg(u')$, for all non-leaf vertices $u^{'}$ in $\mathcal{S}({Y})$. Since $$\displaystyle{\tau(B^{'},x)=\frac{1}{1+x\left(\deg(r(B^{'}))-1\right)}},$$$$\displaystyle{\tau(B^{u'}_L,x)=\frac{1}{1+x\left(\deg(u')-2+\tau(B^{*},x)\right)}},\text{ for some branch $B^{*}$ in $B^{u'}_L$}$$\\ and $\deg(r(B^{'}))\geq\deg(u')$, for all non-leaf vertices $u^{'}$ in $\mathcal{S}({Y})$, we have $\tau(B^{'},x)\leq\tau(B^{u'}_L,x)$, with equality if and only if $\deg(u')=\deg(r(B^{'}))$, $u'$ is a pseudo leaf and $u'$ is not $r(B^{'})$.  Let $G$ be a caterpillar obtained from $\mathcal{S}({Y})$, by removing the edge $wr(B^{'})$ and then adding the edge $w^{'}r(B^{'})$. 
Since $\deg(w^{'})\geq 2=\deg(w)$, $\tau(B_{L}^{u'},x)\geq \tau(B^{'},x)$ and $M(H'-w^{'},x)=M(H'-w,x)$ ($H'$ is the two vertex path with end vertices $w$ and $w^{'}$), then by Lemma \ref{minimal4}, we must have $M(\mathcal{S}({Y}),x)>M(G,x)$.

Recall that $\tau(B,x)\geq\tau(B^{'},x)$. Let $w''$ be a leaf adjacent to $w^{'}$ in $G$. Let $G^{'}$ be a caterpillar obtained from $G$ as follows: if $\deg(v)\geq\deg(w^{'})$ and $M(H-v,x)\leq M(H-w^{'},x)$, then remove the edge $w^{'}w''$ and then add the edge $vw''$. Then by Lemma \ref{lemma needed} we have $$M(G,x)>M(G^{'},x).$$ Otherwise, $\deg(v)<\deg(w^{'})$ or $M(H-v,x)> M(H-w^{'},x)$. If $\deg(v)\leq\deg(w^{'})$ and $M(H-v,x)\geq M(H-w^{'},x)$, then let $G^{*}$ be obtained from $G$ by swapping $B$ and $B^{'}$, and then $G^{'}$ obtained from $G^{*}$ by removing the edge $w^{'}w''$ and then adding the edge $vw''$. By Lemma \ref{minimal1} we have $M(G,x)>M(G^{*},x)$ and by Lemma \ref{lemma needed} we have $M(G^{*},x)>M(G^{'},x)$. If $\deg(v)>\deg(w^{'})$ and $M(H-v,x)\geq M(H-w^{'},x)$, then let $G^{*}$ be obtained from $G$ by flipping $H$,and then $G^{'}$ obtained from $G^{*}$ by removing the edge $w^{'}w''$ and then adding the edge $vw''$ 
. Then by Lemma \ref{minimal1} we have $M(G,x)>M(G^{*},x)$ and by Lemma \ref{lemma needed} we have $M(G^{*},x)>M(G^{'},x)$. If $\deg(v)\leq\deg(w^{'})$ and $M(H-v,x)<M(H-w^{'},x)$, then let $G^{*}$ be obtained by swapping $d$ and $d^{'}+1$, and then obtain $G^{'}$ from $G^{*}$ by removing the edge $w^{'}w''$ and then adding the edge $vw''$. Then by Lemma \ref{minimal1} we have $M(G,x)>M(G^{*},x)$ and by Lemma \ref{lemma needed} we have $M(G^{*},x)>M(G^{'},x)$. Then, by Lemmas \ref{minimal1} and \ref{lemma needed},we must have $M(G,x)>M(G^{'},x)$ and the degree sequence of $G^{'}$ is ${Y}_1$. Hence, by Theorem \ref{minimal theorem}, $$M(\mathcal{S}({Y}),x)>M(G,x)>M(G^{'},x)\geq M(\mathcal{S}({Y}_1),x).$$  Note that the degree sequence of $G$ is not ${Y}_1$ and not ${Y}$.

If ${Y}_1=D$, then we are done. Otherwise, we iterate the process. We set ${Y}={Y}_0$, and if $k$ is a positive integer and ${Y}_k\neq D$, then we construct ${Y}_{k+1}$ in exactly the same way as ${Y}_1$ was constructed from ${Y}$. After a finite number $\displaystyle{J=\frac{1}{2}\sum_{i\in\mathbb{I}}|d_i-y_i|}$ of iterations, we will get the chain $${Y}={Y}_0\prec {Y}_1\prec\dots \prec {Y}_{J-1}\prec {Y}_J=D.$$For any $k\in\{1,\dots ,J-1\}$, we can apply Lemmas \ref{minimal1}, \ref{minimal3} and \ref{lemma needed}, with Theorem \ref{minimal theorem} to $\mathcal{S}({Y}_k)$ as we did above, to deduce that there exists a caterpillar $G_{k+1}$ with degree sequence ${Y}_{k+1}$ such that $$M(\mathcal{S}({Y}_k),x)>M(G_{k+1},x)\geq M(\mathcal{S}({Y}_{k+1}),x).$$Hence $$M(\mathcal{S}({Y}),x)>M(\mathcal{S}({Y}_1),x)>\dots >M(\mathcal{S}({Y}_J),x)=M(\mathcal{S}(D),x).$$
\end{proof}
We can deduce corollaries of Theorems \ref{minimal theorem} and \ref{minimal theorem2} using similar ways as in \cite{[1]}. Below are some examples. We skip the proofs, they follow the following sketch.  Whenever the degree sequence of any element of a set of caterpillars $\mathbb{S}$ is majorised by $D$ and $\mathcal{S}(D)\in \mathbb{S}$, then $M(\mathcal{S}(D),x)\geq M(T,x)$ for all $x>0$ and $T\in\mathbb{S}$.

\begin{cor}\label{minimal corrolary1}
For any Caterpillar $C$ of order $n$ and diameter $m$ $(\leq n-1)$, we have $$M(C,x)\geq M(\mathcal{S}(d,\underbrace{2,\dots ,2}_{m-2}),x)\text{, for all positive $x\in\R$}$$and hence $$Z(C)\geq Z(\mathcal{S}(d,\underbrace{2,\dots ,2}_{m-2}))\qquad\text{ and }\qquad En(C)\geq En(\mathcal{S}(d,\underbrace{2,\dots ,2}_{m-2})),$$with equality if and only if $C\cong\mathcal{S}(d,\underbrace{2,\dots ,2}_{m-2})$, where $d=n-m+1$.
\end{cor}
\begin{cor}\label{minimal corrolary3}
For any Caterpillar $C$ of order $n$ and vertex degree  at most $d\leq n-m+1$, we have $$M(C,x)\geq M(\mathcal{S}(\underbrace{d,\dots ,d}_{k},r),x)\text{, for all positive $x\in\R$}$$and hence $$En(C)\geq En(\mathcal{S}(\underbrace{d,\dots ,d}_{k},r))$$and $$Z(C)\geq Z(\mathcal{S}(\underbrace{d,\dots ,d}_{k},r)),$$with equality if and only if $C\cong\mathcal{S}(\underbrace{d,\dots ,d}_{k},r)$, where $k=\left\lfloor\frac{n-2}{d-1}\right\rfloor$ and\\ $r=\begin{cases}0,&\text{ if }k=\frac{n-2}{d-1}\in\mathbb{Z}\vspace{0.2cm}\\n-\left\lfloor\frac{n-2}{d-1}\right\rfloor(d-1)-1,&\text{ otherwise.}\end{cases}$
\end{cor}

\section*{Acknowledgement}

We would like to thank the South African Department of Higher Education and Training's Future Professors Programme Phase 01 for their financial support for E. O. D Andriantiana; and the German Academic Exchange Service (DAAD) and the South African National Research Foundation (NRF) for their financial support for Xhanti Sinoxolo 

\bibliography{sn-bibliography}

\end{document}